\documentclass{amsart}

\usepackage{amsrefs}
\usepackage{amscd}
\usepackage{amssymb}
\usepackage[cmtip,all]{xy}
\usepackage[mathscr]{eucal}
\usepackage{mathrsfs}
\usepackage{dsfont}

\newtheorem{theorem}{Theorem}[section]
\newtheorem{lemma}[theorem]{Lemma}
\newtheorem{proposition}[theorem]{Proposition}
\newtheorem{corollary}[theorem]{Corollary}

\newtheorem{definition}[theorem]{Definition}

\theoremstyle{remark}
\newtheorem{remark}[theorem]{Remark}

\numberwithin{equation}{section}

\newcommand{\cN}{\mathcal{N}}

\newcommand{\CC}{\mathbb{C}}

\newcommand{\cM}{\mathcal{M}}

\newcommand{\cO}{\mathcal{O}}

\newcommand{\PP}{\mathbb{P}}

\newcommand{\cV}{\mathcal{V}}

\DeclareMathOperator{\Ob}{Ob}
\DeclareMathOperator{\eu}{Eu}

\DeclareMathOperator{\Ch}{Ch}
\DeclareMathOperator{\rk}{rk}

\DeclareMathOperator{\Spec}{Spec}
\DeclareMathOperator{\Grass}{Grass}
\DeclareMathOperator{\mult}{mult}
\newcommand{\vir}{\text{\rm vir}}
\newcommand{\cX}{\mathcal{X}}
\newcommand{\cU}{\mathcal{U}}
\newcommand{\cD}{\mathcal{D}}

\newcommand{\cW}{\mathcal{W}}
\newcommand{\cZ}{\mathcal{Z}}
\newcommand{\cY}{\mathcal{Y}}

\newcommand{\ZZ}{\mathbb{Z}}

\newcommand{\G}{{\mathbb G}}

\newcommand{\QQ}{{\mathbb Q}}

\begin{document}

\title{Pro-Chern-Schwartz-MacPherson class for DM stacks}

\author[Jiang]{Yunfeng Jiang}
\address{Department of Mathematics\\
University of Kansas\\
405 Snow Hall\\
1460 Jayhawk Blvd\\
Lawrence 66045 
\\USA}
\email{y.jiang@ku.edu}

\thanks{}
  
\subjclass[2010]{Primary 14N35; Secondary 14A20}

\keywords{}

\date{}

\begin{abstract}
We generalize the definition of Pro-Chern-Schwartz-MacPherson (Pro-CSM) class of Aluffi for schemes to not necessarily proper DM stacks. The Pro-CSM class of  constructible functions on a DM stack $\cX$ can be similarly defined. 
In the case that $\cX$ is proper, the Pro-CSM class of the Behrend function of $\cX$ is the same as the  Chern-Schwartz-MacPherson (CSM) class for the Behrend function. The integration of this class over $\cX$ gives rise to the weighted Euler characteristic corresponding to the Behrend function, thus proving a conjecture of Behrend. 

\end{abstract}

\maketitle

\section{Introduction}

The Chern-Schwartz-MacPherson (CSM) class for a singular algebraic variety is a generalization of Chern class of smooth varieties. In Section 3 of \cite{MacPherson} MacPherson introduced the local Euler obstruction using Nash blow-up for cycles of (singular) algebraic varieties.  The corresponding characteristic class of the local Euler obstruction is the Chern-Mather class defined by pushforward of the cap product of  Chern class of the Nash tangent bundle with the fundamental class of the Nash blow-up. 
 It turns out that as constructible functions, local Euler obstructions form a basis for the group of constructible functions for an algebraic variety $X$. Hence the constant function $\mathds{1}_{X}$ can be written as a linear combination of local Euler obstructions. Apply the Chern-Mather construction for such a combination one gets CSM class for $X$.

The Chern-Mather class and CSM class for a constructible function  has recently become important in Donaldson-Thomas theory as defined by R. Thomas \cite{Thomas}.  In \cite{Behrend},  Behrend  found that if a proper scheme $X$ admits a symmetric obstruction theory, the virtual count, i.e. the integration of $1$ over the zero dimensional virtual fundamental class 
$[X]^\vir$, is the weighted Euler characteristic $\chi(X, \nu_{X})$. The weight is given by  Behrend function $\nu_X$, which is an integer valued constructible function on $X$.  The Behrend function $\nu_X$ is defined as the Euler obstruction of a canonical cycle $\mathfrak{c}_X\in Z_*(X)$ determined by the scheme $X$.   Applying the Chen-Mather construction one gets a class $\alpha_X\in A_*(X)$ in the Chow group $X$, which Behrend calls Aluffi class.  Behrend also proves in Theorem 1.12 of \cite{Behrend} that the integration of the class $\alpha_X$ over $X$ gives the weighted Euler characteristic $\chi(X, \nu_{X})$  if $X$ is a proper scheme.  Actually Behrend's construction works for any DM stack $\cX$, and he conjectured that for a proper DM stack  $\cX$, the integration of the class $\alpha_{\cX}$ over $\cX$ gives the weighted Euler characteristic $\chi(\cX, \nu_{\cX})$.

There is another method, see Aluffi \cite{Aluffi}, to construct CSM class.  Aluffi's construction actually works for any schemes, not necessarily proper. More precisely Aluffi constructs Pro-CSM class  for non proper schemes, and proves that it is the same as the usual  CSM class if the scheme is proper.    Aluffi also defines a degree map from the  Pro-CSM class  of a scheme to $\mathbb{Z}$ such that the degree is the weighted Euler characteristic.  Aluffi  constructs in the same paper a natural transformation $F\leadsto  \hat{A}_*$ from the functor of  constructible functions from category of schemes to abelian groups. 
In this paper we generalize Aluffi's construction of  Pro-CSM class  to DM stacks.  If a DM stack $\cX$ is proper, the  Pro-CSM class  of $\cX$ is the same as  the CSM class  of $\cX$.
The  Pro-CSM class  of a constructible function on $\cX$ can be similarly defined. In particular, the  Pro-CSM class  of Behrend function $\nu_{\cX}$ of $\cX$ is defined.   In the case that $\cX$ is proper, the  Pro-CSM class  of Behrend function $\nu_{\cX}$ is the same as Aluffi class defined in Section 1.4 of \cite{Behrend}.  Hence the integration of it over $\cX$ gives the weighted Euler characteristic $\chi(\cX, \nu_{\cX})$, proving the conjecture of Behrend. 
Note that this question  has been addressed by Maulik and Treumann in \cite{MT} by generalizing Kashiwara's index theorem for Lagrangian intersection from schemes to orbifolds.

The plan of this paper is as follows.  We define  Pro-CSM group  for DM stacks in Section \ref{Pro_Chow_group}  and   Pro-CSM class in Section \ref{Pro_Chow_class}.   
In Section \ref{natural_transformation_functor}  we prove that there is a natural transformation of functors from the category of DM stacks to abelian groups.   We discuss Behrend function and prove the main theorem in  Section \ref{Behrend_function}.  

\subsection*{Convention:} We work over complex number $\CC$, and take $\QQ$ coefficients for Pro-Chow groups and Chow groups.

\subsection*{Acknowledgement}

The author would like to thank Professors Paolo Aluffi, Kai Behrend, Jun Li and Richard Thomas for their encouragement and valuable discussions.  Thanks to Andrew Kresch and Andrei Mustata for email correspondence about weighted blow-ups. 
The author sincerely thanks the referee for carefully reading the paper and  for many valuable comments about polishing the paper. 
This paper is motivated by a joint work with Richard Thomas about virtual signed Euler characteristics.   
This work is partially supported by NFGRF, University of Kansas, and a Simons Foundation Collaboration Grant 311837.

\section{Pro-Chow group}\label{Pro_Chow_group}

\subsection{}

Let $\cX$ be a DM stack which is separated and finite type  over complex number $\CC$, with quasi-projective coarse moduli space.  We denote by $A_*(\cX)$  the Chow group of $\cX$ with $\QQ$-coefficients in the sense of Vistoli
\cite{Vistoli}.  Let $\cX\to X$ be the canonical map to its coarse moduli space. 
Then there is an isomorphism 
$$A_*(\cX)\cong A_*(X).$$
From Vistoli \cite{Vistoli}, $A_*$ is a functor from category of DM stacks to abelian groups, covariant with proper maps of DM stacks.

\subsection{}

Similar to Aluffi \cite{Aluffi}, we define a Pro-Chow functor $\hat{A}_*(\cdot)$ from DM stacks to abelian groups, covariant with 
respect to regular maps for DM stacks.

Let $\cU$ be a DM stack (maybe not proper). Denote $\mathfrak{U}$ to be the category of maps
$$i: \cU\to \cX_i$$
such that $\cX_i$ is a proper DM stack over $\CC$. A morphism between 
$i: \cU\to \cX_i$ and 
$j: \cU\to \cX_j$
is given by a commutative diagram:
\[
\xymatrix{
\cU\ar[r]^{i}\ar[dr]_{j}& \cX_i\ar[d]^{\pi}\\
&\cX_j
}
\]
where $\pi$ is a proper morphism of DM stacks. 

\begin{definition}
An embedding $i: \cU\to \overline{\cU}$ is a \textbf{closure} if $\overline{\cU}$ is complete and $\cU$ is open and dense in 
$\overline{\cU}$.
\end{definition}

\begin{lemma}
Closures of DM stacks form a small cofinal subcategory $\overline{\mathfrak{U}}$ of $\mathfrak{U}$. 
\end{lemma}
\begin{proof}
The proof is similar to Aluffi  Lemma 2.1 of \cite{Aluffi}. The only difference is that we should use stacky or weighted blow-up for DM stacks. We give a brief review below.

Let $i: \cU\to \cX_i$ be a morphism from a DM stack $\cU$ to a proper DM stack
$\cX$.  In \cite{Rydh}, Rydh generalizes Nagata's result of compactification theorem for schemes to DM stacks. 
So if $j: \cU\to \overline{\cU}$ is a fixed closure, then one can construct the following diagram:
\[
\xymatrix{
&\hat{\cU}\ar[d]\ar[dr]^{\hat{i}}&\\
\cU\ar@{^{(}->}[r]^{j}\ar@{^{(}->}[ur]^{\hat{j}}\ar@/_2pc/[rr]_{i}&\overline{\cU}\ar@{-->}[r]^{\overline{i}}&\cX_i, 
}
\]
where $\hat{i}: \hat{\cU}\to \cX_i$ is a proper morphism constructed by weighted blow-up along the complement of $\cU$, and $\hat{j}: \cU\hookrightarrow \hat{\cU}$ is also a closure.  As in Section 8 of \cite{Rydh}, the weighted blow-up is a process combining root stack construction on divisors and ordinary blow-ups. A root stack construction, roughly spewking, is a method on how to put stacky structure on divisors. This also gives a reason why ordinary blow-ups are not enough for the construction for 
DM stacks. Then $\hat{i}$ can be taken as a morphism from $\hat{j}: \cU\to \hat{\cU}$ to 
$i: \cU\to \cX_i$.  So the closures are cofinal in $\mathfrak{U}$.
\end{proof}
\begin{remark}
Taking coarse moduli spaces of the above diagram gives rise to the diagram in the proof of Lemma 2.1 in \cite{Aluffi}. 
\end{remark}

\begin{definition}
The Pro-Chow group of $\cU$ is defined as the inverse limit of the system:
$$\hat{A}_*(\cU)=\varprojlim_{\cX_i\in \Ob(\overline{\mathfrak{U}})}A_*(\cX_i).$$
\end{definition}

\begin{remark}
From the definition, a class $\alpha\in \hat{A}_*(\cU)\cong \hat{A}_*(U)$ is given by the choice of a class $\alpha_i\in A_*(\cX_i)$
for any proper DM stack $\cX_i$ so that $i: \cU\to \cX_i$
is compatible with proper pushforward. 
\end{remark}

\subsection{}\label{pushforward_ProChow}

Let $f: \cU\to \cV$ be a morphism of DM stacks.  Any assignment $\cV\to \cX_i$ induces a morphism  $\cU\to \cV\to\cX_i$. 
So a compatible class in $\hat{A}_*(\cU)$ defines a class in $\hat{A}_*(\cV)$. So there is a pushforward
$$f_*: \hat{A}_*(\cU)\to \hat{A}_*(\cV)$$
and 
$$(f\circ g)_*=f_*\circ g_*,$$
if $g: \cV\to \cW$ is another morphism. 

To define a global pro-Chow class on a DM stack $\cX$, we need to have the so called \textbf{good local data} defined in 
Section 3 of \cite{Aluffi}.  Let $\cU$ be a nonsingular DM stack, and let 
$$i: \cU\to \overline{\cU}$$
be a closure of $\cU$.  
Consider the diagram:
\begin{equation}\label{diagram:coarse:moduli}
\xymatrix{
\cU\ar[r]^{}\ar[d]_{}& \overline{\cU}\ar[d]^{}\\
U\ar[r]&\overline{U}
}
\end{equation}
where $U$ and $\overline{U}$ are the coarse moduli spaces of $\cU$ and $\overline{\cU}$, respectively. 

We say $i$ is a \textbf{good closure} of $\cU$ if 
$\overline{i}: U\to \overline{U}$
 is a good closure on the coarse moduli spaces in the sense of Section 2.4 of \cite{Aluffi}, i.e. $\overline{U}\setminus U$ consists of simple normal crossing divisors and $\overline{\cU}$ is nonsingular. 

\begin{remark}
In the Diagram (\ref{diagram:coarse:moduli}), 
the DM stack $\overline{\cU}$ can be constructed from $\overline{U}$ by root constructions, see Example 2.4.5 in \cite{Cadman}. 
\end{remark}

The following result is a generalization of Proposition 2.5 of \cite{Aluffi}. 
\begin{proposition}
\begin{enumerate}
\item 
Defining a class $\alpha\in \hat{A}_*(\cU)$ is equivalent to assigning 
$\alpha_i\in A_*(\cX_i)$ for any good closure $i: \cU\to \cX_i$ satisfying
\[
\xymatrix{
\cU\ar[r]^{}\ar[dr]_{}& \cX_i\ar[d]^{\pi}\\
&\cX_j
}
\]
where $\pi$ is a (stacky or) weighted blow-up of $\cX_j$, and 
$\alpha_j=\pi_*(\alpha_i)$;
\item  Two elements $\alpha=\beta\in \hat{A}_*(\cU)$ if and only if $\alpha_i=\beta_i$ for all good closures $i$. 
\end{enumerate}
\end{proposition}
\begin{proof}
Let $\pi: \cU\to U$ be the canonical map to its coarse moduli space. 
Then the result is true for the scheme $U$, see Proposition 2.5 in \cite{Aluffi}. Since 
$$\pi_*: A_*(\cU)\stackrel{\cong}{\longrightarrow} A_*(U)$$
is an isomorphism, from diagram $(1)$ and 
$$\pi_*: A_*(\cX_i)\stackrel{\cong}{\longrightarrow} A_*(X_i)$$
a class in $A_*(U)$ determines a class in $\cU$.

Then the result just follows from the weak factorization of birational maps of DM stacks due to D. Bergh and D. Rydh \cite{Bergh},  \cite{BerghRydh}, which generalizes the weak factorization theorem of AKMW \cite{AKMW} to orbifolds. 
\end{proof}

\section{Pro-Chow class of DM stacks}\label{Pro_Chow_class}

\subsection{}

Let $\cU$ be a nonsingular DM stack. Let $i: \cU\to \overline{\cU}$ be a good closure so that $\cD=\overline{\cU}\setminus \cU$ is a divisor with simple normal crossing. 

\begin{definition} \label{Defn_Chern_Log}
Define
$$c_{\cU}^{\overline{\cU}}:=c(\Omega_{\overline{\cU}}^1(\log \cD)^{\vee})\cap [\overline{\cU}]\in A_*(\overline{\cU}).$$
Here $\Omega_{\overline{\cU}}^1(\log \cD)$ denotes the bundle of differential one forms with logarithmic poles along 
$\cD$. 
\end{definition}

\begin{definition}
A \textbf{good local data} for the DM stack $\cX$ is given by:
\begin{enumerate}
\item 
A decomposition
$$\cX=\bigcup_{\alpha}\cU_{\alpha},$$ 
where $\cX$ is the disjoint union of the DM stacks $\cU_\alpha$, and $\cU_\alpha$ is nonsingular and irreducible. 
\item A class $\{\cU_\alpha\}\in \hat{A}_*(\cU_\alpha)$ for all 
$\alpha$ such that the class $\{\cX\}:=\sum_{\cU_\alpha}i_{\cU_{\alpha}*}\{\cU_\alpha\}$ is well-defined. 
\end{enumerate}
\end{definition}

\begin{remark}
That the class $\{\cX\}$ is well-defined means that the class $\{\cX\}$ is independent of the decomposition of $\cX$ into disjoint union of open substacks. 
\end{remark}

\subsection{}

In this section we prove the following result:

\begin{proposition}\label{good_local_data}
The data $\{c_{\cU}^{\overline{\cU}}\}$ defined in Definition \ref{Defn_Chern_Log} gives a good local data for the DM stack $\cX$. 
\end{proposition}

We make the following set up: there is a diagram
\begin{equation}\label{key_diagram_computation}
\xymatrix{
\cV\ar[rr]\ar[dd]_(.7){} & & \overline{\cV}\ar'[d][dd]^{\pi} &\\
& E\ar[rr]^(.3){}\ar[dd]^(.7){}\ar[ul] & & F\ar[dd]^(.7){\rho}\ar[ul]\\
\cU\ar'[r]^(.7){}[rr] & & \overline{\cU} &\\
& \cZ\ar[ul]\ar[rr]^{} & & \cW\ar[ul]_{w},
}
\end{equation}
where 
\begin{enumerate}
\item  $\cW$ is a nonsingular closed irreducible substack of $\overline{\cU}$ meeting $\cD$ with normal crossing.
\item $\cZ=\cW\cap \cU$. If $\cZ\neq \emptyset$, $\cZ\to \cW$ is a good closure.
\item $\pi: \overline{\cV}\to \overline{\cU}$ is the weighted blow-up along $\cW$.
\item $F:$ exceptional divisor $\pi^{-1}(\cW)$.
\item  $E=\pi^{-1}(\cZ)=F\cap \cV$, so that $E\to F$ is a good closure. 
\end{enumerate}

\begin{lemma}\label{Key_lemma}
$$c_{\cU}^{\overline{\cU}}=\pi_*c^{\overline{\cV}}_{\cV\setminus E}+w_*c_{\cZ}^{\cW}.$$
\end{lemma}
\begin{proof}
The proof is similar to Aluffi, except that we work with smooth DM stacks. 
The main points in Aluffi's proof in Proposition 4.3 in \cite{Aluffi} is two lemmas in \cite{Aluffi2}:

\begin{lemma}\label{formula_12}
\begin{enumerate}
\item  $$\pi_{*}(c(T_{\overline{\cV}})\cap [\overline{\cV}])=c(T_{\overline{\cU}})\cap [\overline{\cU}]+(d-1)\cdot c(T_{\cW})\cap [\cW];$$
\item  $$\rho_{*}(c(T_{F})\cap [F])=d\cdot c(T_{\cW})\cap [\cW],$$
\end{enumerate}
\end{lemma}
where $d$ is the Euler characteristic of the weighted projective space of the fibre $\rho: F\to \cW$. 

\begin{proof}
We first prove the formula $(2)$, which is from pushforward formula for the smooth morphism $\rho: F\to \cW$.
The morphism $\rho: F\to \cW$ is a weighted projective fibration associated with the normal bundle 
$N:=N_{\cW/\overline{\cU}}$.  As in Page 32 of \cite{MM}, the weighted projective fibration 
$F$ is the quotient stack 
$$[N_{\cW/\overline{\cU}}\setminus \cW/\CC^*],$$
where $\CC^*$ acts on the normal bundle by weights
$\{w_n\}$. Let 
$\{\cN_n\}_{n\geq 1}$ be the filtration of $N$ under the $\CC^*$
action and define 
$Q_n=\cN_{w_n}/\cN_{w_{n+1}}$ to be the sub bundle such that the $\CC^*$ acts with 
weight $w_n$.  Then from Proposition 2.11 of \cite{MM} we have:
$$c(T_{F})=\rho^* c(T_{\cW})\cdot \prod_{n} c(Q_n\otimes \mathcal{O}_{\overline{\cV}}(1)^{\otimes w_n}).$$
So applying projection formula as in Proposition 3.7 in \cite{Vistoli} we get the result, since 
the top Chern class of  $\prod_{n} c(Q_n\otimes \mathcal{O}_{\overline{\cV}}(1)^{\otimes w_n})$ is the Euler class of the fibre weighted projective space and its integral gives the Euler characteristic of the weighted projective space, see Proposition 1.6 in \cite{Behrend}. 

Now we prove  Formula $(1)$ using the formula in $(2)$.  The DM stack $F$ is the exceptional divisor of 
$\pi$.  Then Proposition 2.12 in \cite{MM} proves:
$$c(T_{\overline{\cV}})=\pi^* c(T_{\overline{\cU}})\cdot \frac{(F+1)\cdot \prod_{n}(p(Q_n(-w_n F)))}{\prod_np(Q_n)},$$
where $p(Q_n)=c(Q_n)$ but taken as a class pullback from $\overline{\cU}$. 
Denote by $P^{w}$ the fibre weighted projective space. 
Then applying pushforward by $\pi_*$ (note that product with $F$ means doing intersection with $F$), and note that 
the integration of $\prod_nw_n c_1(\mathcal{O}(1))^{\dim(P^w)}$  over $P^w$ is $1$, 
we have
$$\pi_*(c(T_{\overline{\cV}}))= c(T_{\overline{\cU}})+(d-1)(c(T_{\cW})\cap [\cW]),$$
where $d$ is the Euler characteristic of the fibre weighted projective space of $\rho$. 
\end{proof}

\begin{corollary}\label{Key_lemma_cor}
Let $\cD_j, j\in J$ be nonsingular hypersurfaces of $\overline{\cU}$ meeting with normal crossings, and let 
$\widetilde{\cD}_i$ be the proper transform of $\cD_j$ in $\overline{\cV}$.  Assume that at least one of the 
$\cD_j$ contains $\cW$. Then
$$\pi_{*}\left(\frac{c(T_{\overline{\cV})}}{(1+F)\prod_{j\in J}(1+\widetilde{\cD}_j)}[\overline{\cV}]\right)=\frac{c(T_{\overline{\cU}})}{\prod_{j\in J}(1+\cD_j)}[\overline{\cU}].$$
\end{corollary}
\begin{proof}
This is from Formulas (1),  (2) in Lemma \ref{formula_12} and the proof in Lemma 3.8 (5) of \cite{Aluffi2}.
Note that Lemma 3.8 (5) of \cite{Aluffi2} only deals with schemes, but the arguments only involve the class of divisors and Chern class of tangent bundles, which both works for smooth DM stacks as in the intersection theory of Vistoli \cite{Vistoli}. 
We omit details. 
\end{proof}

We continue the proof of main Lemma, which can be divided into two cases.

Case I: $\cZ= \emptyset$. Then
$$c_{\cU}^{\overline{\cU}}=\frac{c(T_{\overline{\cU}})}{\prod_{j\in J}(1+\cD_j)}[\overline{\cU}],$$
and 
$$c^{\overline{\cV}}_{\cV\setminus E}=\frac{c(T_{\overline{\cV})}}{(1+F)\prod_{j\in J}(1+\widetilde{\cD}_j)}[\overline{\cV}].$$
Then the formula is from Corollary \ref{Key_lemma_cor}. 

Case II: $\cZ\neq \emptyset$. i.e. $\cW$ is not contained in any component of $\cD$. 
We have 
$\widetilde{\cD}_j=\pi^{-1}(\cD_j)$. 
We calculate:
\begin{align*}
\pi_{*}\left(\frac{c(T_{\overline{\cV}})}{\prod_{i}(1+\widetilde{\cD}_i)}[\overline{\cV}]\right)&=\frac{1}{\prod_{i}(1+\cD_i)}\cap \pi_{*}(c(T_{\overline{\cV}})\cap [\overline{\cV}])\\
&=\frac{1}{\prod_{i}(1+\cD_i)}\cap \pi_{*}(c(T_{\overline{\cU}})\cap [\overline{\cU}])+(d-1)\cdot w_{*}(c(T_{\cW})\cap [\overline{\cW}])
\end{align*}
by Formula (1) in Lemma \ref{formula_12}.
So 
$$\pi_{*}c_{\cV}^{\overline{\cV}}=c_{\cU}^{\overline{\cU}}+(d-1)\cdot w_{*}c_{\cZ}^{\cW}.$$
On the other hand,  similar calculation as in Proposition 4.3 of Aluffi \cite{Aluffi} gives:
$$c_{\cV\setminus E}^{\overline{\cV}}=c_{\cV}^{\overline{\cV}}-j_{*}c_{E}^{F}.$$
Since $\rho_{*}c_{E}^{F}=d\cdot c_{\cZ}^{\cW}$ from Formula (2) in Lemma \ref{formula_12}.  These formulas together imply the result. 
\end{proof}

\textbf{Proof of Proposition \ref{good_local_data}:}
This basically follows from the following formula:
\begin{equation}\label{decomposition_formula}
\{\cU\}=z_{*}\{\cZ\}+i_{*}\{\cU\setminus \cZ\}
\end{equation}
where $z: \cZ\hookrightarrow \cU$ and $i: \cU\setminus \cZ\hookrightarrow \cU$ are embeddings and 
$z_{*}$, $i_{*}$ are push forwards defined in Section \ref{pushforward_ProChow}.  
Formula (\ref{decomposition_formula}) is proved in Lemma \ref{Key_lemma}.
If there is a decomposition $\cX=\cup_{\alpha}\cU_{\alpha}$ and $\sum_{\alpha}i_{\alpha*}\{\cU_{\alpha}\}$ is a class in 
$\widehat{A}_{*}(\cX)$, we claim that $\sum_{\alpha}i_{\alpha*}\{\cU_{\alpha}\}$ is independent of the decomposition. 
Any two decompositions have a refinement. So every element $\cU$ of one decomposition is a finite disjoint union of elements from the other,  i.e. if 
$$\cU=\cV_1\cup\cV_2\cup\cdots\cup \cV_r$$
where $\cV_j$ are nonsingular, and $\cV_j$ is closed in $\cV_1\cup\cV_2\cup\cdots\cup \cV_j$, then   
$$\{\cU\}=i_{1*}\{\cV_1\}+\cdots+i_{r*}\{\cV_r\}$$
is just from Formula (\ref{decomposition_formula}). 
$\square$

\subsection{}
\begin{definition}
For a DM stack $\cX$, the Pro-CSM class of $\cX$ is defined by
$$\{\cX\}\in \widehat{A}_{*}(\cX)$$
\end{definition}

\begin{proposition}\label{inclusion-exclusion}
The class $\{\cX\}$ satisfies the following property:
if $\cX=\cup_{j\in J}\cX_j$ is a finite union of substacks, then
$$\{\cX\}=\sum_{\emptyset\neq I\subset J}(-1)^{|I|+1}i_{I*}\{\cX_{I}\}$$
where $\cX_{I}=\cap_{i\in I}\cX_i$ and $i_{I}: \cX_{I}\hookrightarrow \cX$ is the inclusion. 
\end{proposition}
\begin{proof}
It suffices to prove the case $|J|=2$. Then 
$$\{\cX_1\cup\cX_2\}=\{\cX_1\}+\{\cX_2\}-\{\cX_1\cap \cX_2\}$$
just follows from definition. 
\end{proof}

\subsection{}

Let $\cX$ be a DM stack.  A $\QQ$-valued constructible function $\phi: \cX\to \QQ$ is a compactly supported function such that it is 
constant along locally closed substacks of $\cX$.  It is well-known that the group of integer-valued constructible functions on a scheme $X$ is generated by characteristic functions of subschemes. Since every DM stack is generically the quotient of a scheme by a finite group, Corollarie 6.11 in \cite{LMB}, the constructible functions of $\cX$ are well-defined.  A detail treatment of constructible function using Grothendieck site can be found in Section 2 and 3 in \cite{MT}. 

Let $F(\cX)$ be the group of $\QQ$-valued constructible functions on $\cX$. Let $\phi\in F(\cX)$ be a constructible function. We write:
$$\phi=\sum_{\cZ}n_{\cZ}\mathds{1}_{\cZ},$$
where $\cZ$ is a locally closed substack of $\cX$,  $n_{\cZ}\in \QQ$ and  $\mathds{1}_{\cZ}$ is the constructible function
$$
\mathds{1}_{\cZ}(p)=
\begin{cases}
1, & p\in \cZ;\\
0, & p\notin \cZ.
\end{cases}
$$

\begin{definition}\label{ProChow_constructible_function}
The Pro-Chow class $\{\phi\}$ of a constructible function $\phi$ of $\cX$ is defined by:
$$\{\phi\}=\sum_{\cZ}n_{\cZ}\cdot i_{\cZ*}\{\cZ\}$$
where $i_{\cZ}: \cZ\hookrightarrow \cX$ is the inclusion. 
\end{definition}
\begin{remark}
The Pro-Chow class $\{\phi\}$ is a class in the Pro-Chow group $\widehat{A}_*(\cX)$. From inclusion-exclusion Proposition \ref{inclusion-exclusion} it  is independent of the decomposition of the constructible function. 
\end{remark}

\subsection{}

We briefly review the Euler characteristic for DM stacks in Section 1.3 of \cite{Behrend}, which generalizes the Euler characteristic for schemes. 

From Section 1.3 in \cite{Behrend}, the Euler characteristic with compact support
$\chi$ is a $\QQ$-valued function on isomorphism classes of pairs $(\cX, f)$, where $\cX$ is a DM stack and $f$ is 
a constructible function on $\cX$.  It satisfies the following properties:
\begin{enumerate}
\item if $\cX$ is separated and smooth, $\chi(\cX, \mathds{1})=\chi(\cX)$ is the usual Euler characteristic of $\cX$, see \cite{Behrend2};
\item $\chi(\cX, f+g)=\chi(\cX, f)+\chi(\cX, g)$;
\item if $\cX$ is the disjoint union of a closed substack $\cZ$ and its open complement $\cU$, then 
$\chi(\cX, f)=\chi(\cU, f|_{\cU})+\chi(\cZ, f|_{\cZ})$;
\item $\chi(\cX\times\mathcal{Y}, f\boxdot g)=\chi(\cX, f)\cdot \chi(\mathcal{Y}, g)$;
\item if $\cX\to \cY$ is a finite \'stale morphism of degree $d$, then $\chi(\cX, f|_{\cX})=d\cdot \chi(\cY, f)$, for any 
constructible function $f$ on $\cY$. 
\end{enumerate}
Then we take $\chi(\cX, \mathds{1}_{\cX})=\chi(\cX)$ the  Euler characteristic of $\cX$. 

Applying the pushforward map in Section \ref{pushforward_ProChow} to the structure map $\cU\to \Spec k$, we get a well-defined \textbf{degree} for each 
$\alpha\in \widehat{A}_*(\cU)$:
$$\int\alpha\in A_*(\Spec k)=\QQ.$$
Thus every constructible function $\phi\in F(\cX)$ gives rise to 
$$\int:  F(\cX)\to \QQ$$
by
$$\phi\mapsto \int_{\cX}\{\phi\}.$$
\begin{proposition}
\begin{enumerate}
\item If $\cX$ is proper and nonsingular, 
$\{\cX\}=c(T\cX)\cap [\cX]$ and 
$$\int_{\cX}\{\cX\}=\chi_{}(\cX)$$
is the Euler characteristic of $\cX$. 
\item Let $\cX$ be proper.  If $\phi\in F(\cX)$, then 
$$\int_{\cX}\{\phi\}=\chi(\cX, \phi)$$
is the weighted Euler characteristic of $\cX$. 
\end{enumerate}
\end{proposition}
\begin{proof}
(1) is the Gauss-Bonnet theorem, see Proposition 1.6 in \cite{Behrend}. 
For (2), from properties of weighted Euler characteristic in the beginning of this section, and Proposition \ref{inclusion-exclusion}, both 
$\int\{\cdot\}$ and $\chi(\cdot)$ satisfy inclusion-exclusion, so we only need to prove  
$\int_{\cX}\{\cX\}=\chi(\cX, \mathds{1}_{\cX})$ of $\cX$ is compact and nonsingular. This is (1).
\end{proof}

\section{The natural transformation functor}\label{natural_transformation_functor}

\subsection{} 

In Section \ref{Pro_Chow_class} we define Pro-Chow class and Pro-CSM class for a DM stack $\cX$. 
Moreover, we define Pro-CSM class for any constructible function $\phi$ on $\cX$. 

Let $f: \cX\to \cY$ be a morphism of DM stacks. 
Let 
$\phi=\sum_{\cZ}n_{\cZ}\cdot \mathds{1}_{\cZ}\in F(\cX)$ be a constructible function on 
$\cX$, where $\cZ\hookrightarrow \cX$ is the inclusion of a locally closed substack, and $n_\cZ\in\QQ$. 
Define
$$f_{*}: F(\cX)\to F(\cY)$$
by
$$f_{*}(\phi)(p)=\sum_{\cZ}n_{\cZ}\cdot \int \{f^{-1}(p)\cap \cZ\},$$
where $\{f^{-1}(p)\cap \cZ\}$ is the Pro-Chow class of $f^{-1}(p)\cap \cZ$, and $\int \{f^{-1}(p)\cap \cZ\}$ is the 
Euler characteristic of $f^{-1}(p)\cap \cZ$. 
Since we work on morphisms of non proper DM stacks, we  first prove that 
$F$ is a functor from category of DM stacks to abelian groups of constructible functions. 

\begin{theorem}\label{Covariance}
Let $f: \cX\to \cY$, $g: \cY\to\cZ$ be morphisms of DM stacks. Then
$$(g\circ f)_{*}=g_{*}\circ f_{*}.$$
\end{theorem}
\begin{proof}
We provide the basic method on how to prove this result for DM stacks, generalizing the proof of 
\cite{Aluffi}.   The idea is to decompose the DM stack into nonsingular irreducible substacks, and do a similar 
argument as in  Section 5 in \cite{Aluffi}.  
The idea of Grothendieck site as developed in  \cite{MT} can be used prove this result, but we use a direct method here. 
 
A \emph{gentle} morphism $f: \cU\to \cV$ between two nonsingular DM stacks $\cU$ and $\cV$ is a smooth and surjective morphism. Furthermore there exists a proper DM stack $\underline{\cU}$ and a smooth and surjective morphism 
$\underline{f}: \underline{\cU}\to \cV$ such that:
\begin{enumerate}
\item $\cU$ is an open dense subset of $\underline{\cU}$, and $\underline{f}|_{\cU}=f$;
\item the complement $\mathcal{H}=\underline{\cU}\setminus \cU$ is a divisor with normal crossings and nonsingular components $\mathcal{H}_i$;
\item  letting $\mathcal{H}_{I}=\cap_{i\in I}\mathcal{H}_i$ (so that $\mathcal{H}_{\emptyset}=\underline{\cU}$) and 
$\mathcal{H}_{I}$ is nonsingular, then the restriction 
$$\underline{f}|_{\mathcal{H}_{I}}: \mathcal{H}_{I}\to \cV$$
is proper, smooth and surjective. 
\end{enumerate}
If $f: \cU\to\cV$ is gentle, then $\chi_{f}=\int\{f^{-1}(p)\}$ is independent of $p\in \cV$.  This is from Lemma \ref{key_lemma} and a similar argument as in Lemma 5.5 in \cite{Aluffi} for smooth DM stacks. Hence 
$f_*\{\cU\}=\chi_f\cdot \{\cV\}$, which can be proved using inclusion-exclusion and Lemma \ref{key_lemma} again. 
So if $f: \cU\to \cV$ and $g: \cV\to\cW$ are two gentle morphisms, then 
$$\chi_{g\circ f}=\chi_g\cdot \chi_f.$$
Now to prove the theorem, it suffices to prove the two push forwards agree on the characteristic functions of substacks of $\cX$. 
We need to decompose $f: \cX\to \cY$ and $g: \cY\to\cZ$ into gentle morphisms.  A DM stack is  locally a scheme by a finite group action, hence using the morphism to its coarse moduli space (which is a scheme),  morphisms
$f: \cX\to \cY, g: \cY\to \cZ$ can be decomposed into:
$$\cX=\sqcup_{\alpha, i,j}\cU_{\alpha, i,j}, \quad \cY=\sqcup_{\alpha, i}\cV_{\alpha, i},\quad \cZ=\sqcup \cW_\alpha$$
such that 
$$f_{\alpha,i,j}:=f|_{\cU_{\alpha,i,j}}: \cU_{\alpha,i,j}\to \cV_{\alpha,i}, \quad g_{\alpha,i}:=g|_{\cV_{\alpha,i}}: \cV_{\alpha,i}\to \cW_{\alpha}$$
are gentle.  
Also for $$\cU_{\alpha,i,j}\stackrel{f_{\alpha,i,j}}{\longrightarrow}\cV_{\alpha,i}\stackrel{g_{\alpha,i}}{\longrightarrow} \cW_{\alpha}$$
we have
$$\chi_{\alpha ij}=\chi_{\alpha ij}^\prime\cdot \chi_{\alpha i},$$
where $\chi_{\alpha ij}^\prime,  \chi_{\alpha i}, \chi_{\alpha ij}$ are the corresponding fiberwise degrees for 
$f_{\alpha, i,j}, g_{\alpha,i}$ and $g_{\alpha,i}\circ f_{\alpha, i,j}$ respectively. So a direct calculation gives:
$$g_{*}(f_{*} \mathds{1}_{\cU_{\alpha,i,j}})=g_{*}(\chi_{\alpha ij}^\prime\cdot \mathds{1}_{\cV_{\alpha,i}})=\chi_{\alpha ij}^\prime\cdot \chi_{\alpha i}\cdot \mathds{1}_{\cW_\alpha}=(g\circ f)_{*}(\mathds{1}_{\cU_{\alpha,i,j}}).$$
\end{proof}

Then the assignment:
$$c_*: F\leadsto\widehat{A}_{*}$$ by
$$\phi\mapsto c_*(\phi)=\{\phi\}$$
is a transformation of functors.
In particular,  $c_*(\mathds{1}_{\cX})=c_*(\cX)\in \widehat{A}_*(\cX)$ is the Pro-CSM class of $\cX$. 
In this section we prove that $c_*$ is a natural transformation of  functors. 
The naturality theorem is:
\begin{theorem}(Naturality)\label{natural_transformation}
The transformation functor $c_*:  F\to \widehat{A}_*$ is a natural transformation of covariant functors from 
the category of DM stacks to the category of abelian groups. 
\end{theorem}
\begin{remark}
In particular, $c_*$ sends the constant function $\mathds{1}_{\cX}$ for a smooth DM stack $\cX$ to $c(T\cX)\cap [\cX]$.
\end{remark}

The key point to prove Theorem \ref{Covariance} and Theorem \ref{natural_transformation} is the following Lemma:

\begin{lemma}\label{key_lemma}
Let $f: \cU\to\cV$ be a proper, smooth, surjective map of nonsingular DM stacks. Then 
$$f_{*}(\{\cU\})=\chi_{f}\cdot \{\cU\}$$
where $\chi_{f}=\int\{f^{-1}(p)\}$ and any $p\in \cV$. 
\end{lemma}
\begin{proof}
From the definition of Pro-Chow group class $\{\cU\}$, it is sufficient to prove that for any diagram of closures:
$$
\xymatrix{
\cU\ar[r]^{i}\ar[d]_{f}& \cX\ar[d]^{g}\\
\cV\ar[r]^{j}& \cY
}
$$
where $i$ and $j$ are good closures, $f$ is smooth, $f$ and $g$ are proper and surjective, we have:
$$g_{*}c_{\cU}^{\cX}=\chi_{f}\cdot c_{\cV}^{\cY},$$
i.e.
\begin{equation}\label{key_formula_1}
g_{*}(c(\Omega_{\cX}^{1}(\log \cD)^\vee)\cap [\cX])=\chi_{f}\cdot c(\Omega_{\cY}^{1}(\log E)^\vee)\cap [\cY],
\end{equation}
where $\cD, E$ are the complement $\cX\setminus \cU$, $\cY\setminus \cV$ respectively. 

We prove (\ref{key_formula_1}) using \textbf{graph construction} of MacPherson on DM stacks, which we construct below. For the morphism
$g: \cX\to \cY$, we consider:
$$dg: g^*\Omega_{\cY}^{1}(\log E)\to \Omega_{\cX}^{1}(\log \cD).$$
Since $g$ is smooth over $\cU$, $dg$ is injective over $\cU$. We will apply graph construction to 
$dg$. The formula (\ref{key_formula_1}) is equivalent to the following:
\begin{equation}\label{key_formula_2}
g_{*}(c(\Omega_{\cX}^{1}(\log \cD))\cap [\cX])=\widetilde{\chi}_{f}\cdot c(\Omega_{\cY}^{1}(\log E))\cap [\cY],
\end{equation}
where $\widetilde{\chi}_f$ is the same as $\chi_f$ up to a sign.
\end{proof}

The next several subsections will finish the proof of Formula  (\ref{key_formula_2}).

\subsection{}
We generalize the graph construction of MacPherson to DM stacks.  
The construction is similar to the scheme case of MacPherson in \cite{MacPherson}, except we work 
with smooth DM stacks and cotangent bundles. 
Note that in \cite{KKP}, Kim, Kresch and Pantev already used the graph construction  for the compatibility of obstruction theories which involves DM stacks. 

Let $m=\dim(\cX), n=\dim(\cY)$. Then $dg$ gives a rational morphism:
$$\gamma: \cX\times \PP^1\dashrightarrow \G:=\Grass_{n}(g^*\Omega_{\cY}^{1}(\log E)\oplus \Omega_{\cX}^{1}(\log \cD))$$
by
$$x\times \{(\lambda:1)\}\mapsto ~ \mbox{the graph of~}\frac{1}{\lambda}dg \mbox{~at~} x.$$
The indeterminates of $\gamma$ are contained in $\cD\times 0\subset \cX\times 0:=\cX\times \{[0:1]\}$. So we have the following diagram:
\begin{equation}\label{graph_diagram}
\xymatrix{
\widetilde{\cX\times \PP^1}\ar[d]^{\pi}\ar[dr]^{\widetilde{\gamma}}\ar@/_3pc/[ddd]_{p}&\\
\cX\times\PP^1\ar@{-->}[r]^{\gamma}\ar[d]^{\rho}&\G\\
\cX\ar[d]^{g}&\\
\cY
}
\end{equation}
where $\widetilde{\cX\times \PP^1}$ is the graph of $\gamma$, (i.e. the stacky blow-up of $\cX\times \PP^1$ along the indeterminacies). 
For stacky blow-ups, see \cite{MM} and \cite{Rydh}. 
Then 
$$[\pi^{-1}(\cX\times \{\infty\})]=[\pi^{-1}(\cX\times\{0\})].$$
We can express the pre image as:
$$\pi^{-1}(\cX\times\{0\})=\widetilde{\cX}\bigcup \cup_ir_i\Gamma_i,$$
where $\widetilde{\cX}$ is the proper transform of $\cX\times \{0\}$ and $\Gamma_i$ are exceptional divisors and $r_i$ are multiplicities. 
All of these results are from Example 18.1.6 of Fulton \cite{Fulton}, which can be generalized to smooth DM stacks.
The basic intersection theory of Fulton \cite{Fulton} holds for DM stacks and thanks for A. Kresch for communication about this theory. 

Let $\mathcal{Q}$ be the universal rank-$m$ quotient bundle over the Grassmannian bundle $\G$. 
Then 
$$c(\widetilde{\gamma}^*\mathcal{Q})\cap [\pi^{-1}(\cX\times \{\infty\})]
=c(\widetilde{\gamma}^*\mathcal{Q})\cap \left([\widetilde{\cX}]+\sum_ir_i[\Gamma_i]\right)\in A_*(\cX\times\PP^1)$$
First we have:
$$p_{*}(c(\widetilde{\gamma}^*\mathcal{Q})\cap [\pi^{-1}(\cX\times \{\infty\})])=g_{*}(c(\Omega_{\cX}^{1}(\log \cD))\cap[\cX]),$$
as may be verified by chasing Diagram (\ref{graph_diagram}), since over 
$\cX\times \infty:=\cX\times\{[1:0]\}$, $\gamma$ acts on the section corresponding to 
$g^*\Omega_{\cY}^1(\log E)\oplus 0$.

\begin{lemma}
$$p_{*}(c(\widetilde{\gamma}^*\mathcal{Q})\cap [\widetilde{\cX}])=\widetilde{\chi}_f\cdot c(\Omega_{\cY}^{1}(\log E))\cap [\cY].$$
\end{lemma}
\begin{proof}
The restriction $\widetilde{\gamma}|_{\widetilde{\cX}}=\widetilde{\gamma}^\prime$ factors through 
$$\G^\prime:=\Grass_{n}(\Omega_{\cX}^{1}(\log \cD))\cong \Grass_{n}(\Omega_{\cX}^{1}(\log \cD)\oplus 0)\subset \G.$$
Over $\G^\prime$, $\mathcal{Q}=g^*\Omega_{\cY}^{1}(\log E)\oplus \mathcal{Q}^\prime$. So
$$c(\widetilde{\gamma}^*\mathcal{Q})\cap [\widetilde{\cX}]=c(p^*\Omega_{\cY}^{1}(\log E))\cdot c(\widetilde{\gamma}^{\prime*}\mathcal{Q}^\prime)\cap [\widetilde{\cX}].$$
Since $\mathcal{Q}^\prime$ is the universal rank-$(m-n)$ quotient bundle over $\G^\prime$, it is the cotangent bundle of the fibres over points of $\cV$.  Taking pushforward along $p$ the result follows. 
\end{proof}
Hence
\begin{equation}\label{formula:need:proof}
g_{*}(c(\Omega_{\cX}^{1}(\log \cD))\cap [\cX])-\widetilde{\chi}_{f}\cdot c(\Omega_{\cY}^{1}(\log E))\cap [\cY]=\sum_{i}r_i\cdot p_{*}(c(\mathcal{Q})\cap [\Gamma_i]).
\end{equation}

\begin{lemma}
For any component $\Gamma$ in the exceptional divisors in $\widetilde{\cX\times\PP^1}$,
$$p_{*}(c(\mathcal{Q})\cap [\Gamma])=0.$$
\end{lemma}
\begin{proof}
Consider the following diagram:
$$
\xymatrix{
\Gamma\ar[r]^{\widetilde{\gamma}}\ar[d]_{\sigma}\ar@/_2pc/[dd]_{p}&\G\ar[d]\\
\cZ\ar@{^{(}->}[r] \ar[d]&\cX\ar[d]^{g}\\
\cW\ar@{^{(}->}[r] &\cY
}
$$
Here $\cZ$ and $\cW$ are the images of $\Gamma$ in $\cX$ and $\cY$, respectively, and also 
$\cZ\subset \cD,  \cW\subset E$. 
We assume $\cW\subseteq E_i$ for $i\leq s$, and $\cW\nsubseteq E_i$ for $i>s$. Denote by
$$E_{\underline{s}}=E_1\cap\cdots\cap E_s.$$
Let $\mathcal{S}_{\Gamma}$ and $\mathcal{Q}_{\Gamma}$ be the pullback to $\Gamma$ of the universal subbundle and quotient bundle over 
$\G$.  Then we have the exact sequence:
$$0\to \mathcal{S}_{\Gamma}\longrightarrow \sigma^*(g^*\Omega_{\cY}^{1}(\log E)\oplus\Omega_{\cX}^{1}(\log \cD))|_{\cZ}\longrightarrow
\mathcal{Q}_{\Gamma}\to 0.$$

On the other hand, we have the residue exact sequence:
$$0\to \Omega_{E_{\underline{s}}}^{1}|_{\cW}\longrightarrow \Omega_{\cY}^{1}(\log E)|_{\cW}\longrightarrow
\cO_{\cW}^{\oplus s}\oplus (\oplus_{i>s}\cO_{\cZ\cap E_i})\to 0.$$
Consider the following commutative diagram:
\begin{equation}\label{}
\xymatrix{
&0\ar[d]&&&&\\
&N_{\cW}E_{\underline{s}}^*\ar[d]\ar[dr]&&&&\\
0\ar[r]&\Omega_{E_{\underline{s}}}^{1}|_{\cW}\ar[d]\ar[r]&\Omega_{\cY}^{1}(\log E)|_{\cW}\ar[r]\ar[d]&\cO_{\cW}^{\oplus s}\oplus (\oplus_{i>s}\cO_{\cZ\cap E_i})\ar[r]&0\\
0\ar[r]&\Omega_{\cW}^{1}\ar[r]\ar[d]&\mathcal{T}\ar[d]\ar[ur]&&\\
&0&0&&
}
\end{equation}
where the exact sequence in the middle row is the residue exact sequence and $\mathcal{T}$ is the quotient of 
$\Omega_{\cY}^{1}(\log E)|_{\cW}$ by $N_{\cW}E_{\underline{s}}^*$.  Note that 
$\rk(\mathcal{T})>\dim(\cW)$.

Hence we have the following diagram:
$$
\xymatrix{
0\ar[r]&\mathcal{S}_{\Gamma}\ar[r]^{}&\sigma^*(g^*\Omega_{\cY}^{1}(\log E)\oplus\Omega_{\cX}^{1}(\log \cD))|_{\cZ}\ar[r]\ar[d]
&
\mathcal{Q}_{\Gamma}\ar[d]\ar[r]&0\\
&p^*N_{\cW}E_{\underline{s}}^*\ar[r]^{\psi}&p^*\Omega_{\cY}^{1}(\log E)|_{\cW}\ar[r]&p^*\mathcal{T}\ar[r]&0
}
$$
where the map $\varphi$ is the first nontrivial arrow in the top row and  the image of $\varphi$ is contained in $\mbox{Im}(\psi)$. 
So the induced map 
$\mathcal{Q}_{\Gamma}\twoheadrightarrow p^*\mathcal{T}$ is surjective. 

So now we have a morphism 
$p: \Gamma\to \cW$ and a vector bundle $\mathcal{Q}_{\Gamma}$ over $\Gamma$ of rank 
$\leq \dim(\Gamma)$. Also 
$\mathcal{Q}_{\Gamma}\twoheadrightarrow p^*\mathcal{T}$ is surjective, where 
$\mathcal{T}$ is a coherent sheaf over $\cW$. 
Then a similar argument in Lemma 6.5 of \cite{Aluffi} implies that
$$p_{*}(c(\mathcal{Q}_{\Gamma})\cap [\Gamma])=0.$$
The proof uses a birational morphism $\nu: \widetilde{\cW}\to \cW$  of DM stacks such that 
$\nu^*\mathcal{T}$ is locally free on $\widetilde{\cW}$. By weighted blow-up on the locus that $\mathcal{T}$ is not 
locally free we get the DM stack $\widetilde{\cW}$. 
\end{proof}

Comparing to (\ref{formula:need:proof}),  this concludes the  proof of Lemma \ref{key_lemma}.

\subsection{Proof of Theorem \ref{natural_transformation}}
The key Lemma \ref{key_lemma} helps us to prove that 
the pushforward of Pro-CSM class for gentle morphisms, as in the proof of Theorem \ref{Covariance}. 

If $f: \cX\to \cY$ is a morphism of DM stacks, then using gentle morphisms, and letting 
$$\cX=\sqcup_{\alpha,i} \cU_{\alpha,i}, \quad \cY=\sqcup_{\alpha}\cV_{\alpha}$$
be the decomposition of $\cX$ and $\cY$ such that 
$$f|_{\cU_{\alpha,i}}: \cU_{\alpha,i}\to \cV_{\alpha}$$
is gentle.  Then by a similar argument as in Lemma 5.8 in \cite{Aluffi} we have
$$f_{*}(\mathds{1}_{\cX})=\sum_{\alpha\in A}\chi_{\alpha}\cdot \mathds{1}_{\cV_{\alpha}},$$
$$f_*(\{\cX\})=\sum_{\alpha\in A}\chi_{\alpha}\cdot \{\cV_{\alpha}\}.$$
These results are enough for proving 
$$f_{*}(\{\phi\})=\{f_{*}(\phi)\}$$
for any constructible function $\phi\in F(\cX)$, since 
by linearity of $f_*$, it suffices to prove this for characteristic function of any substacks of $\cX$.
$\square$

\section{The Pro-CSM class of Behrend function}\label{Behrend_function}

\subsection{}
Let $\cX$ be a DM stack, proper or not.  Behrend \cite{Behrend} introduces a constructible function 
$\nu_{\cX}$, which is defined as follows:
There is a canonical integral cycle $\mathfrak{c}_{\cX}\in Z_*(\cX)$, such that 
if \'etale locally (or Zariski locally if $\cX$ is a scheme) there is an open chart 
$U\to \cX$ and an embedding 
$U\hookrightarrow M$ into a smooth DM stack $M$, 
$$\mathfrak{c}_{\cX}|_{U}=\sum_{i}\mult(C_i)(-1)^{\dim(\pi(C_i))}[\pi(C_i)],$$
where $C_i$ is the irreducible component of the normal cone 
$C_{U/M}$ and 
$\pi: C_{U/M}\to U$ is the projection; $\mult(C_i)$ is the multiplicity of $C_i$ at generic point.

\begin{definition}(Behrend)
$$\nu_{\cX}=\eu(\mathfrak{c}_{\cX}),$$
where $\eu$ is the local Euler obstruction of MacPherson in Section 1.2 of \cite{Behrend}, which generalizes the local Euler obstruction in the case of varieties  in Section 3 of \cite{MacPherson}. 
\end{definition}

The Behrend function $\nu_{\cX}$ is a constructible function on $\cX$. From Section \ref{Pro_Chow_class} the Pro-Chow class of $\nu_\cX$ is well-defined. 
\begin{definition}
The Pro-CSM class for $\nu_{\cX}$ is defined by:
$$c^{PSM}(\nu_{\cX})=\{\nu_{\cX}\}\in \widehat{A}_*(\cX).$$
\end{definition}
\begin{remark}\label{remark:pro-chow:nuX}
By choosing a good stratification for $\cX$, we may write 
$$\nu_{\cX}=\sum_{\cZ}a_{\cZ}\cdot \mathds{1}_{\cZ},$$
i.e. the Euler obstruction of $\mathfrak{c}_{\cX}$ is constant on $\cZ$. 
Then $$\{\nu_{\cX}\}=\sum_{\cZ}a_{\cZ}\cdot i_{\cZ*}\{\cZ\}.$$
Taking degree we get 
$$\int_{\cX}\{\nu_{\cX}\}=\chi(\cX, \nu_{\cX}),$$
which is the weighted Euler characteristic of $\cX$. 
\end{remark}

\subsection{}
From  this section  we assume that $\cX$ is proper, then the Pro-Chow group of $\cX$ is just the general Vistoli Chow group of 
$\cX$.  We introduce Chern-Mather class for $\cX$. 

Let $Z\subset \cX$ be a prime cycle.  The Nash blow-up $\mu: \widehat{Z}\to Z$ and the Nash tangent bundle $TZ$  are introduced by MacPherson  in Section 2 of \cite{MacPherson} if $\cX$ is a scheme; and by Behrend in Section 1.2 and Section 1.4 of \cite{Behrend} if 
$\cX$ is a DM stack. 

\begin{definition}
The Chern-Mather class $c^{M}(Z)$ is:
$$c^{M}([Z])=\mu_{*}(c(TZ)\cap [\widehat{Z}]),$$
where $c^{M}([Z])\in A_*(\cX)$. 
\end{definition}
As in Behrend \cite{Behrend}, let  $c_0^{M}(Z)$ be the degree zero part $c_0^{M}: Z_*(\cX)\to A_{0}(\cX)$. 
The Euler obstruction $\eu(Z)$ is a constructible function $\cX\to \ZZ$. The weighted Euler characteristic $\chi(\cX, \eu(Z))$ is given by
$$\sum_in_i\cdot \chi(\eu(Z)^{-1}(i)).$$
Our cycle $\mathfrak{c}_{\cX}$ is a linear  combination of prime cycles in $\cX$. We define 
$$\alpha_{\cX}=c^{M}(\mathfrak{c}_{\cX})\in A_*(\cX)$$
to be the Chern-Mather class of the canonical cycle $\mathfrak{c}_{\cX}$, which Behrend calls the \textbf{Aluffi class}.

\subsection{}\label{Lagrangian_intersection}
We give an explanation of the weighted Euler characteristic in terms of Lagrangian intersections. 
We fix an embedding $\cX\to \cM$ of the DM stack $\cX$ into a smooth DM stack $\cM$ of dimension $n$. The following commutative diagram is due to Behrend,  
Diagram (2) of \cite{Behrend}.
\begin{equation}\label{Key_Diagram}
\xymatrix{
Z_*(\cX)\ar[r]^{\eu}_{\cong}\ar[dr]_{c^{M}_0}&F(\cX)\ar[r]^{\mbox{Ch}}_{\cong}\ar[d]^{c_0^{SM}}&\mathfrak{L}_{\cX}(\Omega_{\cM})\ar[dl]^{I(\cdot, [\cM])}\\
&A_0(\cX)&
}
\end{equation}
where $Z_*(\cX)$ is the group of integral cycles of $\cX$, $F(\cX)$ is the group of constructible functions on $\cX$, and $\mathfrak{L}_X(\Omega_M)$ is the subgroup of $Z_n(\Omega_{\cM})$ generated by the conic Lagrangian prime cycles supported on $\cX$. 
The maps $c^{M}_0$, $c_0^{SM}$ and $I(\cdot, [\cM])$ represent the degree zero Chern-Mather class, degree zero CSM class and the Lagrangian intersection with zero section of $\Omega_{\cM}$, respectively.  Note that in \cite{Behrend}, the notation of Lagrangian intersection with zero section is denoted by $0^{!}_{\Omega_{\cM}}(\cdot)$. 

We briefly explain the horizontal morphisms in the diagram.  The first map is the Euler obstruction $\eu$ and it gives an isomorphism from $Z_*(\cX)$ to $F(\cX)$. 

In Section 4.1 of \cite{Behrend} Behrend defined the following isomorphism of groups:
\begin{equation}\label{cycle_Lagrangian}
L:  Z_{*}(\cX) \to \mathfrak{L}_{\cX}(\Omega_{\cM})
\end{equation}
which is given by
$$Z\mapsto (-1)^{\dim(Z)}N^*_{Z/\cM},$$
where $N^*_{Z/\cM}$ is the closure of the conormal bundle of smooth part of $Z$ inside $\cM$. 
Conversely there is an isomorphism:
\begin{equation}\label{Lagrangian_cycle}
\pi:  \mathfrak{L}_{\cX}(\Omega_{\cM})\to Z_{*}(\cX) 
\end{equation}
which is given by
$$V\mapsto (-1)^{\pi(V)}\pi(V),$$
where $\pi: V\to \cX$ is the projection.  
Then the morphism $\mbox{Ch}$ is defined by the isomorphism $\eu$ and the morphism $L$ defined above.

\subsection{}
We prove the main theorem in this section. 
In the case that $\cX$ is a proper scheme, the integration $\int_{\cX}\alpha_{\cX}$ is the weighted Euler characteristic 
$\chi(\cX, \nu_{\cX})$, which is MacPherson's index theorem in \cite{MacPherson}.  
Note that only degree zero of $\alpha_{\cX}$ contributes. 

\begin{theorem}\label{main1}
Let $\cX$ be a proper DM stack. Then we have
$$\int_{\cX}\alpha_{\cX}=\chi(\cX, \nu_{\cX})$$
\end{theorem}
\begin{proof}
In Definition \ref{ProChow_constructible_function}  of Section \ref{Pro_Chow_class} we define the Pro-CSM class $\{\nu_\cX\}$ for the Behrend function 
$\nu_{\cX}$. By Remark \ref{remark:pro-chow:nuX}, taking integration 
$\int_{\cX}\{\nu_{\cX}\}=\chi(\cX, \nu_{\cX})$, we get the weighted Euler characteristic. 
Actually to get the number only degree zero part $\{\nu_{\cX}\}_{0}\in A_0(\cX)$ contributes. 
So it suffices to prove that the Chow classes 
 $\{\nu_{\cX}\}_0=(\alpha_{\cX})_0$ in $A_{0}(\cX)$.
 
 Recall that our DM stack $\cX$ is quasi-projective, meaning that there is a locally closed embedding to a smooth DM stack 
 $\cM$ with projective coarse moduli space.  Then the cycle 
 $$\mathfrak{c}_{\cX}=\sum_{C_i}(-1)^{\mult(C_i)}\mult(C_i)[\pi(C_i)],$$
 where $$\pi: C:=C_{\cX/\cM}\to \cX$$
 is the projection from the normal cone to $\cX$, $C_i$ are all the irreducible components of $C$, and $[\pi(C_i)]$ is a prime cycle in 
 $Z_*(\cX)$. 
 Since there are finite number of irreducible components of $C$, the cycle $\mathfrak{c}_{\cX}$ is a linear combination of finite prime cycles in $Z_*(\cX)$.  Let $V:=[\pi(C_i)]$ be one such prime cycle.  Recall that 
$\nu_{\cX}=\eu(\mathfrak{c}_{\cX})$ and $\alpha_{\cX}=c^{M}(\mathfrak{c}_{\cX})$. So to prove that 
$(\alpha_{\cX})_0=\{\nu_{\cX}\}_0\in A_0(\cX)$, it suffices to prove that 
$c^M_0(V)=\{\eu(V)\}_0\in A_0(\cX)$. 

According to Lemma 1 in \cite{Kennedy}, Chern-Mather class $c^M(V)$ is the class coming from conormal cycle 
$N^*_{V/\cM}\subset \Omega_{\cM}|_{\cX}$.  Let $\dim(V)=p$ and 
$\mu: \widetilde{\cM}\to \cM$ be the Grassmannian of rank $p$ quotient of $\Omega_{\cM}$ and 
$$\mu: \hat{V}\to V$$
the closure inside $\widetilde{\cM}$ of the canonical rational section $V\dasharrow \widetilde{\cM}$. 
This is exactly the construction of Nash blow-up $\hat{V}$ of $V$.  Hence we have an exact sequence on $\hat{V}$:
$$0\rightarrow N|_{\hat{V}}\longrightarrow \mu^*\Omega_{\cM}|_{\hat{V}}\longrightarrow Q|_{\hat{V}}\rightarrow 0,$$
where $Q$ is the universal quotient bundle on $\widetilde{\cM}$, and $N$ is the kernel of the universal quotient map
$\mu^*\Omega_{\cM}\to Q$. Restricting to $\hat{V}$ we have the above exact sequence. 
From Proposition 4.6 in \cite{Behrend}, there is a cartesion diagram:
\[
\xymatrix{
\hat{V}\ar[r]^{\mu}\ar[d]_{0}& V\ar[r]\ar[d]_{0}& \cM\ar[d]_{0}\\
N|_{\hat{V}}\ar[r]^{\eta}& N_{V/\cM}^*\ar[r]& \Omega_{\cM},
}
\]
where $\eta: N|_{\hat{V}}\to N_{V/\cM}^*$ is a proper birational map of integral stacks. 
Hence from the commutative diagram (\ref{Key_Diagram}), 
$\Ch(\eu(V))=(-1)^{p}N_{V/\cM}^*$ and we have 
$$c_0^M(V)= I((-1)^{p}N_{V/\cM}^*, [\cM]).$$
 
On the other hand, from Section 8 in  \cite{Ginzburg} let $V=\cup_{\alpha}S_\alpha$ be a stratification  of $V$, where $S_\alpha$ is nonsingular and the local Euler obstruction $\eu(V)$ is constant on $S_\alpha$.  
Then 
$$(-1)^{p}N_{V/\cM}^*=\sum_{\alpha}n_{\alpha}\cdot (-1)^{\dim(S_{\alpha})}N_{S_{\alpha}/\cM}^*$$
is a linear combination of characteristic cycles of $S_{\alpha}$. 
Then the Lagrangian intersections are equal, i.e. 
$$I((-1)^{p}N_{V/\cM}^*, [\cM])=\sum_{\alpha}n_{\alpha}\cdot I((-1)^{\dim(S_{\alpha})}N_{S_{\alpha}/\cM}^*, [\cM]).$$
From Diagram (\ref{Key_Diagram}) and Section \ref{Lagrangian_intersection},  $\eu, \mbox{Ch}$ and $L$ are all isomorphisms.  Hence the local Euler obstruction $\eu(V)$ is actually 
$\eu(V)=\sum_{\alpha}n_{\alpha}\eu(S_{\alpha})$. 
The strata $S_{\alpha}$ are nonsingular, and we can take stratification such that $S_{\alpha}$ are also compact, then 
$\sum_{\alpha}n_{\alpha}\cdot I((-1)^{\dim(S_{\alpha})}N_{S_{\alpha}/\cM}^*, [\cM])$ is the same as 
$$\sum_{\alpha}n_{\alpha}i_{\alpha*}(\{S_{\alpha}\}_0)=\{\eu(V)\}_0 \in A_0(\cX),$$
where $i_{\alpha}: S_{\alpha}\hookrightarrow V$ is the inclusion. 
So $c^M_0(V)=\{\eu(V)\}_0 \in A_0(\cX)$ and $(\alpha_{\cX})_0=\{\nu_{\cX}\}_0\in A_0(\cX)$. The theorem is proved.  
\end{proof}
\begin{remark}
From Diagram (\ref{Key_Diagram}),  the Chern-Mather class $\alpha_{\cX}$ of the canonical cycle $\mathfrak{c}_{\cX}$ is the same as the Chern-Schwartz-MacPherson class of the  Behrend function $\nu_{\cX}$.  
One may generalize the construction of the natural transformation of functors  in \cite{MacPherson} 
from category of varieties to DM stacks.  
Then the result  $\alpha_{\cX}=c^{PSM}(\nu_X)$ comes from the naturality of the transformation $F$ of functors in Theorem \ref{natural_transformation}, since $c^{PSM}(\nu_X)$ and $\alpha_X$ all satisfy the pushforward properties. 
\end{remark}

\begin{remark}
Theorem 5.5 was conjectured by Behrend in Remark 1.13 of \cite{Behrend}. 
\end{remark}

Finally we present Theorem \ref{main1} as Lagrangian intersection:

\begin{theorem}
We have:
$$\int_{\cX}\alpha_{\cX}=\chi(\cX, \nu_{\cX})=I(\Ch(\nu_{\cX}), [\cM]),$$
where $I(\Ch(\nu_{\cX}), [\cM])$ means the Lagrangian intersection number. 
\end{theorem}
\begin{proof}
From Diagram (\ref{Key_Diagram}), $\Ch(\nu_{\cX})\subset \Omega_{\cM}$ is the characteristic cycle of 
$\nu_{\cX}$, which conic Lagrangian. 
The first equality is Theorem \ref{main1} and the second follows from  Theorem 4.5 and Theorem 5.3 in \cite{MT}. 
\end{proof}

\bibliographystyle{amsplain}

\end{document}